\documentclass[11pt]{article}
\usepackage{amsmath}
\usepackage{amsthm}
\usepackage{amsfonts}
\usepackage{amssymb}
\usepackage{amsrefs}
\usepackage{epsfig}
\usepackage{graphicx}

\usepackage{color}
\definecolor{grey}{rgb}{0.8,0.8,0.8}
\newcommand{\edit}[2]{#2}

\newcommand{\ra}{\ensuremath{\rightarrow}}

\newcommand{\dist}{\operatorname{dist}}

\newcommand{\half}{\frac{1}{2}}

\renewcommand{\i}{\ensuremath{\infty}}

\newcommand{\dx}{\;dx}

\renewcommand{\d}{\;d}

\newcommand{\ha}[1]{\frac{#1}{2}}

\newcommand{\fall}{\qquad\text{for all }}

\newcommand{\abs}[1]{\left|#1\right|}
\newcommand{\pth}[1]{\left(#1\right)}
\newcommand{\set}[1]{{\left\{#1\right\}}}
\newcommand{\at}[2]{{{\left.{#1}\right|}_{#2}}}

\newcommand{\bra}[1]{{\left[{#1}\right]}}

\newcommand{\cl}[1]{\overline{#1}}

\newcommand{\no}[1]{\left\|#1\right\|}

\newcommand{\al}{\ensuremath{\alpha}}
\newcommand{\be}{\ensuremath{\beta}}

\newcommand{\de}{\ensuremath{\delta}}
\newcommand{\vp}{\ensuremath{\varphi}}
\newcommand{\la}{\ensuremath{\lambda}}
\newcommand{\e}{\ensuremath{\varepsilon}}
\newcommand{\ve}{\ensuremath{\varepsilon}}

\newcommand{\ta}{\ensuremath{\theta}}

\newcommand{\om}{\ensuremath{\omega}}

\newcommand{\R}{\ensuremath{\mathbb{R}}}

\newcommand{\Rn}{\ensuremath{{\mathbb{R}^n}}}

\newcommand{\bmt}{\begin{pmatrix}}
\newcommand{\emt}{\end{pmatrix}}

\newcommand{\bcs}{\begin{cases}}
\newcommand{\ecs}{\end{cases}}

\newcommand{\seq}[2]{\ensuremath{\{{#1}_{#2}\}_{{#2} = 1}^{\infty}}}

\newcommand{\pd}[2]{\frac{\partial {#1}}{\partial {#2}}}

\numberwithin{equation}{section}

\newtheorem{thm}{Theorem}[section]

\newtheorem{lemma}[thm]{Lemma}

\newtheorem{corollary}[thm]{Corollary}
\newtheorem{proposition}[thm]{Proposition}
\newtheorem{definition}[thm]{Definition}
\newtheorem{remark}[thm]{Remark}

\newcommand{\tin}{\text{in }}

\begin{document}

\title{Viscosity Solutions for the two-phase Stefan Problem}
\author{Inwon C. Kim\footnote{Email: ikim@math.ucla.edu}\ { }and Norbert Po\v{z}\'{a}r\footnote{Email: npozar@math.ucla.edu}\\ Department of Mathematics, UCLA\\
Los Angeles, CA 90095-1555}
\maketitle
\begin{abstract}
We introduce a notion of viscosity solutions for the two-phase Stefan problem, which incorporates possible existence of a mushy region generated by the initial data. We show that a comparison principle holds between viscosity solutions, and investigate the coincidence of the viscosity solutions and the weak solutions defined via integration by parts.  In particular, in the absence of initial mushy region, viscosity solution is the unique weak solution with the same boundary data.
\footnote{This is a preprint of an article whose final and definitive form has been published in Communications in Partial Differential Equations, 2010. Copyright Taylor \& Francis. Communications in Partial Differential Equations is available online at:  http://www.informaworld.com/LPDE}
\end{abstract}

\section{Introduction}

The classical two-phase Stefan problem models the evolution of temperature in a material with two distinct phases (solid and liquid). The solid-liquid phase transition occurs at a given constant temperature, which we set at zero.  The problem is described by an \emph{enthalpy} function $h$, which represents the internal energy density in each phases (see \textsc{Meirmanov} \cite{M} and \textsc{Oleinik et al.} \cite{OPR}). More precisely, let $\Omega\subset \Rn$ be a bounded domain in $\Rn$ with a $C^2$ boundary. Then $h(x,t) : \Omega \times [0,T) \to \R$ solves the following problem:
\begin{align}
\label{eq:StefanProblem}
\tag{ST}
\bcs
h_t = \Delta \chi(h)& \text{ in } \Omega\times (0,T),\\ \\
h(x,0) = h_0(x) & \text{ in } \Omega,\\ \\
\chi(h) = \theta, & \text{ on } \partial \Omega \times (0, T).
\ecs
\end{align}
In this paper we will set $\theta(x,t)$ as a continuous function on $\partial \Omega \times [0,T]$ and $h_0$ as a bounded measurable function with $\chi(h_0)$ continuous on $\cl{\Omega}$, such that $\chi(h_0) = \ta(\cdot,0)$ on $\partial \Omega$. Here $u(x,t):=\chi(h)(x,t)$ describes the evolving temperature, where $\chi$ is defined by
\begin{align}
\label{eq:chi}
\chi(h) = \bcs
h, &  h >0,\\
 0, & -1\leq h \leq 0,\\
h+1, & h<-1.
\ecs
\end{align}

  Note that $h$ is determined by $u$ almost everywhere if $\Gamma(u):=\{u=0\} = \{-1\leq h\leq 0\}$ is of measure zero, but not in general. The interior of zero temperature set $\Gamma(u)$, if exits, is called a {\it mushy region}.  It is well known that mushy region might appear in a Stefan problem with heat sources, see \textsc{Bertsch et al.} \cite{BMP}. When there is no heat source, the mushy region can only be generated from the initial data. Indeed if $|\{-1\leq h_0 \leq 0\}| = 0$ (here $|\cdot|$ denotes the Lebesgue measure in $\R^n$), it is shown in \cite{RB} (see also \cite{GZ} and \cite{A} )
that $|\{-1\leq h(\cdot,t)\leq 0\}|=0$ for all $t>0$ and \eqref{eq:StefanProblem} can be written in terms of the temperature: 
\begin{align}
\tag{ST$u$}
\label{eq:StefanProblemFormal}
\left\{\begin{array}{lll}
u_t - \Delta u = 0 & \hbox{ in }&\set{u > 0} \cup \set{u < 0},\\ \\
V_n = \abs{D u^+} - \abs{D u^-} & \hbox{ on }& \Gamma(u) =\partial\{u>0\}=\partial\{u<0\},\\ \\
u(x,0)=u_0(x) := \chi(h_0(x)) &&\\ \\
u = \ta & \hbox{ on } & \partial \Omega \times (0,T)
\end{array}
\right.
\end{align}
 Here $u^+$ and $u^-$ denotes respectively the limit taken from $\{u>0\}$ and $\{u<0\}$, and  $V_n$ is the outward normal velocity of the free boundary with respect to $\{u>0\}$.

\edit{}{On the other hand, if $|\{-1\leq h_0\leq 0\}|>0$  with no heat source then the mushy region is non-increasing as the solution evolves (see \cite{RB}): in this case the zero temperature set $\Gamma(u)$ may exhibit instability such as infinite speed of propagation as its interior diminishes (see \cite{BKM} where the interface evolution is investigated in the case of one-phase Hele-Shaw flow with initial mushy region).}

For both \eqref{eq:StefanProblem} and \eqref{eq:StefanProblemFormal}, even solutions with initially smooth configuration may develop singularities in finite time, due to merging and splitting of the liquid-solid regions. Thus it is necessary to introduce a generalized notion of  global-in-time solutions.

The first global notion of solutions for \eqref{eq:StefanProblem} has been derived by a weak formulation via integration by parts: see  \textsc{Lady\v{z}enskaja et. al.} \cite{LSU} or \textsc{Friedman} \cite{F}. This approach allows the presence of initial mushy region. 
It was shown there that a unique weak solution exists for bounded measurable $h_0$ and $\ta$. Moreover the temperature $u(x,t)$ turns out to be continuous when $h_0$ and $\ta$ are sufficiently smooth -- see  \textsc{Caffarelli \& Evans} \cite{CE} and \textsc{DiBenedetto} \cite{DB}.  

More recently, the notion of viscosity solutions, first introduced by \textsc{Crandall \& Lions} \cite{CL}, has been proved useful in the qualitative study of free boundary problems, due to its flexibility with nonlinear PDEs and free boundary conditions: see e.g., \cite{ACS1}--\cite{ACS2} and \cite{C} which address a general class of free boundary problems including \eqref{eq:StefanProblemFormal}, and also see \cite{FS} where the heat operator is replaced by a class of uniformly parabolic operators. However, in the two-phase case, the equivalence between viscosity solutions and weak solutions, and in particular the uniqueness of the viscosity solutions of \eqref{eq:StefanProblem}, has not yet been 
investigated, even in the case of \eqref{eq:StefanProblemFormal}. (The only notable instance is \cite{CSalsa} where the authors show a comparison principle between viscosity subsolutions and more regular supersolutions, so-called ``R-supersolutions''. Even thought their result could be used to simplify the proof of Lemma~\ref{th:zGradient}, it is restricted to the heat operator, whereas our method is more robust and extends to a wider class of problems, see Remark~\ref{extension}.) In this paper we address precisely this point.   We introduce a notion of viscosity solutions for \eqref{eq:StefanProblem}, to derive the following theorem: 

\begin{thm}\label{main}
\begin{itemize}
\item[(a)] There exists a viscosity solution of \eqref{eq:StefanProblem} in $Q$. Furthermore, a {\it comparison principle} holds between viscosity solutions and thus  maximal and minimal
viscosity solutions \edit{exists}{exist}. They respectively \edit{corresponds}{correspond} to weak solutions of \eqref{eq:StefanProblem} with maximal and minimal initial enthalpy for given temperature $u_0$.\\
\item[(b)] Any weak solution of \eqref{eq:StefanProblem} in the sense of  \cite{CE} is also a viscosity solution.\\
\item[(c)] If $|\{-1\leq h_0\leq 0\}|=0$, then there is a unique viscosity solution of \eqref{eq:StefanProblemFormal}  which coincides with the weak solution of \eqref{eq:StefanProblem}, and also coincides with the notion of viscosity solutions discussed in \cite{CL} and \cite{ACS1}.
 \end{itemize}
\end{thm}

Previously, one of the authors investigated in \cite{K1} viscosity solutions theory for the one-phase version of \eqref{eq:StefanProblemFormal} \edit{}{(also see \cite{CV} and \cite{BV} for the case of degenerate diffusion)}.  While one-phase Stefan problem is a special case of the two-phase Stefan problem, this work does not immediately apply to the two-phase version: we will point out the differences as we proceed in the proof.  For example, the free boundary  evolution for the two-phase setting is no longer monotone (that is, the positive set of $u$ may shrink or expand over time), which requires a more careful construction of test functions. Moreover, since we are allowing the initial data to have mushy region, $\Gamma(u)$ may have positive measure with infinite speed of propagation, and thus must be treated differently.  \edit{}{Indeed, to the best of our knowledge, Theorem~\ref{main} is among the first results which addresses uniqueness of viscosity solutions for a two-phase free boundary problem.}

Below we give an outline of the paper:

In sections~\ref{sec:problem} and \ref{sec:comparison} we introduce a definition of viscosity solutions for \eqref{eq:StefanProblem}, and prove that viscosity solutions satisfy a comparison principle for strictly separated initial data. Note that this does not immediately yield uniqueness except under some geometric conditions on the initial data (see Corollary~\ref{cor:unique}), however it generates a maximal and minimal viscosity solutions (see Proposition~\ref{prop:maximal}).
These maximal and minimal solutions indeed correspond to the weak solutions with maximal and minimal initial enthalpy.
 
In section~\ref{sec:weakAreViscosity}, we show that  weak solutions are viscosity solutions. Using this result, existence of viscosity solutions is established. Furthermore, using the uniqueness of the weak solutions and barrier arguments near the parabolic boundary of $Q$ (see Lemma~\ref{convergence}), we establish the uniqueness of viscosity solutions when $|\{u_0=0\}|=0$ and $\ta < 0$.  

\begin{remark}\label{extension}
The proof of the comparison principle (Theorem~\ref{th:comparison}) relies only on the comparison principle available for the parabolic operator $u_t - \Delta u$ and monotonicity of the boundary condition. It therefore applies to the generalized version of \eqref{eq:StefanProblemFormal} described in \cite{ACS1}, where the second equation in \eqref{eq:StefanProblemFormal} is replaced by a general free boundary velocity of the type
$$
V=G(x,|Du^+|, |Du^-|, \eta),
$$
where $\eta$ is the spatial (outward) unit normal vector of $\Gamma(u)$, $G(x,a,b,\eta)$ is continuous and 
$$
\frac{\partial G}{\partial a}, -\frac{\partial G}{\partial b} \geq c>0.
$$
And in fact, a general uniformly parabolic operator 
$$
u_t - F(t, x, u, Du, D^2u)
$$ 
discussed in \cite{CIL} can be considered in $\set{u \neq 0}$. That is, we require that $F$ is an uniformly elliptic operator such that
\begin{align*}
\la \no{N} \leq F(t, x, v, p, M +N) - F(t, x, v, p, M) \leq \Lambda \no{N} 
\end{align*}
where $t \in [0,T]$, $x\in \Omega$, $v \in \R$, $p \in \R^n$, $M, N \in S(n)$, $N\geq 0$, and $0 < \la \leq \Lambda$. The use of Hopf's lemma at the end of the proof of Theorem~\ref{th:comparison} can be omitted if we consider the inf-convolution $W$ in \eqref{eq:convolutions} on shrinking balls as in \cite{K1}, which only requires a more demanding notation.

As for existence, a slight modification of Perron's method presented in \cite{K1} combined with Theorem~\ref{th:comparison} yields the existence of a viscosity solution for the generalized version. This approach is not taken here since we are also interested in the coincidence of weak and viscosity solutions.
\end{remark}

\section{The Stefan problem}
\label{sec:problem}

Let $\partial_L Q = \partial \Omega \times [0,T]$ be the lateral boundary of $Q$ and $\partial_P Q = \pth{\Omega\times \set{0}} \cup \partial_L Q$ be the parabolic boundary of $Q$. 

We shall use two notions of generalized solutions for \eqref{eq:StefanProblem}: weak solutions and viscosity solutions.

\begin{definition}
\label{def:weakSolution}
A bounded measurable function $h$ is called the \emph{weak solution} of \eqref{eq:StefanProblem} with initial data $h_0 \in L^\i(\Omega)$ and boundary data $\ta \in L^\i(\partial_L Q)$ if it satisfies (see \cite{M})
\begin{align}
\label{eq:weakSolution}
\int_Q \pth{h \vp_t + \chi(h) \Delta \vp} \dx \d t - \int_{\partial_L Q} \ta \pd{\vp}{\nu} \d s \d t + \int_\Omega h_0(x) \vp(x, 0) \dx = 0
\end{align}
for all $\vp \in W^{2,1}_2(Q)$ vanishing on $\partial_L Q$ and at $t = T$.
\end{definition}

Due to \cites{LSU, F, M}, there exists a unique weak solution $h$ of problem \eqref{eq:StefanProblem}, defined in Definition~\ref{def:weakSolution}, with initial data $h_0$ and boundary data $\ta$. Furthermore a comparison principle holds between weak solutions (\cite{M}): see the proof of Lemma~\ref{th:weakIsViscosity} for more precise statement.

As mentioned in the introduction, the temperature $u(x,t)$ is related to the enthalpy $h$ through $u = \chi(h)$. Formal computations yield that \eqref{eq:StefanProblem} provides the free boundary velocity
\begin{align}
\label{eq:freeBoundaryV}
\tag{V}
\begin{array}{lll}
V_n & = |Du^+|-|Du^-| &\hbox{ on } \partial\{u>0\} \cap\partial\{u<0\};\\ \\
V_n        &\geq |Du^+| &\hbox{ on } \partial\{u>0\}\setminus\partial\set{u<0};\\ \\
\tilde V_n        & \leq - |Du^-| & \hbox{ on } \partial\{u<0\}\setminus\partial\set{u>0},
\end{array}
\end{align}
where $\tilde V_n$ is the outward normal velocity of the free boundary of $\set{u \geq 0}$. The inequalities are due to the (possible) presence of the mushy region and the non-constant distribution of $h(x,t)$.  Definitions~\ref{def1} and \ref{def2} incorporate \eqref{eq:freeBoundaryV}.
\edit{}{
\begin{definition}
\label{def:parabolicNeighborhood}
Set $D \subset \R^{n+1}$ is a \emph{parabolic neighborhood of a point $(x_0,t_0)$} if there is an open set $U$, $(x_0,t_0) \in U$, such that $D = U \cap \set{t \leq t_0}$ and $D \in C^{0,1}$, i.e. the boundary of $D$ is locally a graph of a Lipschitz function in $x$ and $t$. 
\end{definition}
}
\begin{definition}
\label{def:classicalSubsolution}
A continuous function $\phi$ is a \emph{classical subsolution} (in a parabolic neighborhood $\edit{\Omega_{x_0,t_0}:}{D}$) of \eqref{eq:StefanProblem} if 
\begin{enumerate}
\item $\phi$ is $C^{2,1}_{x,t}$ in $\edit{}{D\cap}\cl{\{\phi>0\}}$ and in $\edit{}{D\cap}\cl{\{\phi<0\}}$, \\
\item $ \Gamma(\phi)=\partial\{\phi>0\}=\partial\{\phi<0\}$,\\
\item $|D\phi^+|$ and $|D\phi^-|$ are nonzero on $\edit{}{D\cap}\Gamma(\phi)$ (this makes $\Gamma(\phi)$ smooth),\\
\item $\phi_t - \Delta \phi \leq 0$ in $\edit{}{D\cap}\set{\phi \neq 0}$,\\
\item the outward normal velocity $V_n$ of $\edit{}{D\cap}\Gamma(\phi)$ with respect to $\set{\phi > 0}$ is \edit{}{less than} or equal to $|D\phi^+| - |D\phi^-|$.
\end{enumerate}
Similarly one can define a \emph{classical supersolution} in a parabolic neighborhood.
\end{definition}

\begin{definition}\label{def1}
\begin{itemize}
\item[(a)] A lower semi-continuous function $v$ is a \emph{viscosity supersolution} of \eqref{eq:StefanProblem} in $Q$ if (i) $v(x,0) \geq \chi(h_0)$, (ii) $v \geq \ta$ on $\partial_L Q$, and if (iii) for any classical subsolution $\phi(x,t)$ in \edit{$\Omega_{x_0,t_0}\subset Q$}{a parabolic neighborhood $D \subset Q$} such that $\phi\leq v$ on the parabolic boundary of \edit{$\Omega_{x_0,t_0}$}{$D$}, $\phi\leq v$ in \edit{$\Omega_{x_0,t_0}$}{$D$}.\\
\item[(b)] An upper semi-continuous function $u$ is a \emph{viscosity subsolution} of \eqref{eq:StefanProblem} in $Q$ if (i) $u(x,0)\leq \chi(h_0)$, (ii) $u \leq \ta$ on $\partial_L Q$, and if (iii) for any classical supersolution $\phi(x,t)$ in \edit{$\Omega_{x_0,t_0}\subset Q$}{a parabolic neighborhood $D \subset Q$} such that $\phi\geq u$ on the parabolic boundary of \edit{$\Omega_{x_0,t_0}$}{$D$}, $\phi\geq u$ in \edit{$\Omega_{x_0,t_0}$}{$D$}.
\end{itemize}
\end{definition}
 
 Observe that if $u$ is a viscosity subsolution of \eqref{eq:StefanProblem}, then $-u$ is a supersolution of \eqref{eq:StefanProblem}.
 
\begin{definition}\label{def2}
A continuous function $u$ is a \emph{viscosity solution} of \eqref{eq:StefanProblem} if it is both viscosity subsolution and supersolution.
\end{definition}

\begin{remark}
 Note that above definition of viscosity solutions, unlike that of weak solutions, does not take into account the enthalpy function $h(x,t)$. \edit{Such definition ensures that non-uniqueness occurs}{In particular, we cannot expect uniqueness} when the initial data
$u_0$ satisfies $|\{u_0=0\}|>0$.  A modified definition, incorporating the presence of $h(x,t)$ and therefore yielding a unique representation of the interface evolution, is under investigation by the authors.
\end{remark}

\section{Comparison for strictly separated initial data}
\label{sec:comparison}

A central property for the viscosity solution theory is the \emph{comparison principle}, which is stated in the theorem below. We mention that a corresponding result holds for weak solutions of \eqref{eq:StefanProblem} (see \cite{F}).
As mentioned in the introduction (see Remark~\ref{extension}), the proof presented in this section extends to a general class two-phase free boundary problems with nonlinear free boundary velocity.
\begin{thm}
\label{th:comparison}
Let $u$ and $v$ be respectively a viscosity subsolution and supersolution on $Q$ such that $u < v$ on $\partial_P Q$. Then $u < v$ in $Q$.
\end{thm}

First we introduce several notations. Let $B_r(x,t)$ denote the closed space-time ball with radius $r$ and center at $(x,t)$,
\begin{equation*}
B_r(x,t) = \set{(y,s): \abs{(y,s)-(x,t)} \leq r},
\end{equation*}
and $D_r(x,t)$ the closed space disk with radius $r$ and center at $(x,t)$,
\begin{equation*}
D_r(x,t) = \set{(y,t): \abs{y-x} \leq r}.
\end{equation*}
We also speak of the sets 
$$
\partial D_r(x,t) = \set{(y,t) : \abs{y - x} = r} \quad \hbox{and} \quad D_r^o(x,t) = \set{(y,t) : \abs{y - x} < r}$$
 as the boundary and the interior of $D_r(x,t)$, respectively.

For any $r > 0$, define the sup- and inf-convolutions (see \cite{K1})
\begin{align}
\label{eq:convolutions}
\begin{aligned}
Z(x,t) &= \sup_{B_r(x,t)} U(y,s), \qquad \qquad &U(x,t) &= \sup_{D_r(x,t)} u(y,s),\\
W(x,t) &= \inf_{B_r(x,t)} V(y,s), &V(x,t) &= \inf_{D_r(x,t)} v(y,s).
\end{aligned}
\end{align} 
\edit{Note that $Z$, $U$, resp. $W$, $V$ are also viscosity subsolutions, resp. viscosity supersolutions.}{}
For any $\rho > 0$, define
\begin{align*}
\Omega_{\rho} &= \set{x \in \Omega : \dist(x, \partial \Omega) > \rho},
\end{align*}
The convolutions \eqref{eq:convolutions} are well defined in $\cl{C_r}$,
\begin{align*}
C_r &= \Omega_{2r} \times (r, T -r).
\end{align*} 
Moreover, we fix $r > 0$ small enough so that $Z < W$ on $\partial_P C_{r}$. \edit{}{Note that $Z$, $U$, resp. $W$, $V$ are also viscosity subsolutions, resp. viscosity supersolutions of \eqref{eq:StefanProblem} in $\cl{C_r}$.}

It will be convenient to define $\Phi_r(x,t)$ to be the closed set
\begin{equation}
\label{eq:phir}
\Phi_r(x,t) = \bigcup_{(y,s) \in B_r(x,t)} D_r(y,t).
\end{equation}

\edit{Observe that, by definition, the sets $\set{Z \geq 0}$ and $\set{W \leq 0}$ have interior balls of radius $r$ at each point of their boundaries. These balls are centered on the boundaries of $\set{U \geq 0}$ and $\set{V \leq 0}$, respectively. 
This property of $Z$ and $W$ will be essential in the proofs in this section. It motivates the following definition:}{Lastly, we need the following definition:}
\begin{definition}
Let $E$ be a closed set, let $P$ be a point on the boundary of $E$, $P \in \partial E$, and let $B$ be a closed ball containing $P$. The ball $B$ will be called an \emph{interior ball} of $E$ at $P$ if $B \subset E$, and it will be called an \emph{exterior ball} of $E$ at $P$ if $B^o \cap E = \emptyset$.
\end{definition}
\edit{}{Observe that, by definition, the sets $\set{Z \geq 0}$ and $\set{W \leq 0}$ have interior balls of radius $r$ at each point of their boundaries. These balls are centered on the boundaries of $\set{U \geq 0}$ and $\set{V \leq 0}$, respectively. 
This property of $Z$ and $W$ will be essential in the proofs in this section.}

To prove that $u$ and $v$ stay ordered, we will argue by contradiction: if $u\geq v$ at some time $\tilde t_0$, it must be true for $Z$ and $W$ even earlier. Suppose that there is a time $t_0$ when the functions $Z$ and $W$ first hit,
$$
\edit{t_0 = \sup \set{s \geq r: Z(\cdot,s) < W(\cdot,s)}.}
{t_0 = \sup \set{t \geq r: Z(\cdot,s) < W(\cdot,s) \text{ for all } r < s < t}.}
$$

Denote $P_0 = (x_0, t_0)$ the point of maximum of $Z - W$ at $t = t_0$. Note that $Z(x_0, t_0) \geq W(x_0, t_0)$\edit{. If that were not the case, we would have $Z - W < 0$ at $t = t_0$ and $(Z - W)(x_k,t_k) \geq 0$ for some sequence $(x_k,t_k)$, $x_k \in \Omega$ with $t_k \ra t_0 ^+$. Then there is a converging subsequence $(x_{k_j}, t_{k_j}) \ra (x,t_0)$. Set $\ve = -(Z-W)(x,t_0)/2 > 0$. By upper semi continuity there is $j_0$ such that for all $j \geq j_0$
$$
0 \leq (Z-W) (x_{k_j}, t_{k_j}) < (Z-W) (x,t_0) + \ve \leq -\ve,
$$
a contradiction.}{ because $Z-W$ is upper semi-continuous.} 

 Note that the positive set of $Z$ (and the negative set of $W$) may expand discontinuously in time, in the event that the zero set of $Z$ (and $W$) is of positive measure: this is because our definition only considers test functions which are nondegenerate in both positive and negative phases.  Therefore one should put extra care in describing the location of $P_0$ and the geometry of the free boundaries nearby.
 
\begin{lemma}
\label{lem:continuousExpansion}
The set $\set{u \geq 0}$ cannot expand discontinuously, i.e.
$$
\set{u \geq 0,\ t = T} \subset \cl{\set{u \geq 0,\ t < T}} \qquad \forall T > 0.
$$
Analogous results hold for $\set{v \leq 0}$,  $\set{Z \geq 0}$, $\set{U \geq 0}$, $\set{W \leq 0}$ and $\set{V \leq 0}$.
\end{lemma}
\begin{proof}
 
Suppose that for some $T > 0$ there is a point $P \in \set{u \geq 0,\ t=T}$ such that $P \notin  \cl{\set{u \geq 0,\ t < T}}$. In other words, there exists $\de > 0$ such that $B_\de(P) \cap \set{u \geq 0, \ t<T} = \emptyset$. For simplicity we will translate the coordinates so that $P = (0,0)$. Define $M = \max_Q u$.
Note that the ball $B_\de(P)$ contains a cylinder 
$$
C = \set{(x,t): |x| \leq h + \eta,\ -\ha{\eta} \min \pth{\frac{\eta}{2M}, \frac{h}{n}} \leq t \leq 0}
$$
with $h$ and  $\eta$ small positive constants. The situation is shown in Figure~\ref{fig:continuousExpansion}.
\begin{figure}[t]
\centering
\includegraphics{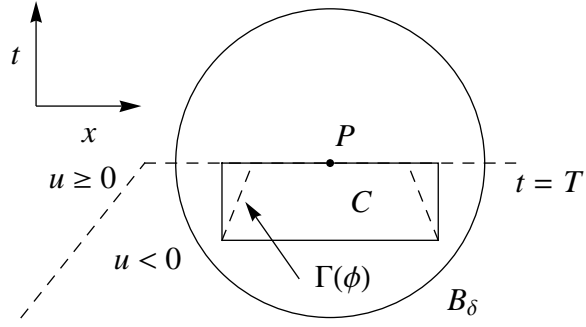}
\caption{Situation at point $P$}
\label{fig:continuousExpansion}
\end{figure}
We will construct a classical 
supersolution of \eqref{eq:StefanProblem} in $C$ that crosses $u$ from above to obtain a contradiction. 

Define $\om =2 \max \pth{\frac{2M}{\eta}, \frac{n}{h}}$.
Consider a radially symmetric function in $C$ defined by the formula
\begin{equation*}
\phi(x,t) = \bcs
a(|x|^2 - (h - \om t)^2) & \text{for } |x| \leq h  -\om t,\\ \\
\frac{2 M}{\eta} \pth{|x| - h + \om t} & \text{for } |x| > h  -\om t,
\ecs
\end{equation*}
where $a>0$ is a small constant specified below. 

Then $\phi$ satisfies
\edit{$$
\begin{array}{llll}
\pth{\partial_t - \Delta} \phi &=&
\frac{2 M}{\eta} \pth{\om - \frac{n-1}{|x|}}>0 &\hbox{ for } |x|>h-\om t\\ &\hbox{ and }&& \\
 & = & 2a\pth{(h-\om t)\om - n} > 0 &\hbox{ for } |x| < h - \om t.
 \end{array}
$$}{\begin{align*}
\pth{\partial_t - \Delta} \phi &=
\bcs
\frac{2 M}{\eta} \pth{\om - \frac{n-1}{|x|}}>0 &\hbox{ for } |x|>h-\om t,\\  \\2a\pth{(h-\om t)\om - n} > 0 &\hbox{ for } |x| < h - \om t.
\ecs
\end{align*}}
Next, note that  the free boundary $\Gamma(\phi)$ at time $t$  is the sphere with radius $h - \om t$, shrinking with normal velocity  $\om$. Therefore we have
$$
V_n = \om > \frac{2M}{\eta} - \abs{D \phi^-}=\abs{D \phi^+ } - \abs{D \phi^-} .
$$
Hence $\phi$ is a classical supersolution of \eqref{eq:StefanProblem} for any $a > 0$. 

Lastly, note that the initial boundary $D = C \cap \set{t = - \eta/\om}$ of the cylinder $C$ is a closed subset of $\set{u < 0}$ and therefore $\max_D u < 0$. We can hence take $a > 0$ small enough to make $\phi > u$ on $D$. 

Thus we have constructed a classical supersolution $\phi$ on $C$ such that $\phi > u$ on the parabolic boundary of $C$, while $\phi < 0 \leq u$ at $P = (0,0)$. 
That shows that $u$ crosses $\phi$ from below inside $C$, which contradicts the definition of $u$.

The conclusion for $\set{v \leq 0} = \set{-v \geq 0}$ can  be obtained by repeating the proof for $-v$.
Corresponding arguments apply to the sets $\set{Z \geq 0}$, $\set{U \geq 0}$, $\set{W \leq 0}$ and $\set{V \leq 0}$ as well.
\end{proof}

\begin{lemma}
\label{lem:p0boundaries}
The point $P_0$ lies on the intersection of boundaries of $\set{Z \geq 0}$ and $\set{W \leq 0}$,
$$
P_0 \in \partial \set{Z \geq 0} \cap \partial \set{W \leq 0},
$$
and 
\begin{align*}
W(P_0) \leq 0 \leq Z(P_0).
\end{align*}
\end{lemma}
\begin{proof}
We know that $W(P_0) \leq Z(P_0)$ from the definition of $P_0$. There are three cases:
\begin{enumerate}
    \item $0 < W(P_0)\leq Z(P_0)$: in this case $W > 0$ in some neighborhood of $P_0$  ($\set{W >0}$ is open as $W$ is a lower-semicontinuous function) and $W$ is a supercaloric function in this neighborhood; 
    this contradicts the fact that $Z$ cannot be crossed from above by any positive caloric function $\phi$\edit{}{ (see \cite{CIL})}. 
\item $W(P_0)\leq Z(P_0) < 0$ : the same as above for $Z$ subcaloric in a neighborhood of $P_0$.
\end{enumerate}
Therefore we arrive at 
\begin{equation}\label{arrive}
W(P_0) \leq 0 \leq Z(P_0).
\end{equation}
Inequality~\eqref{arrive} immediately yields that $P_0 \in \set{Z \geq 0} \cap \set{W \leq 0}$. Now suppose that $P_0$ lies in the interior of the set $\set{Z \geq 0}$. That implies $P_0 \notin \cl{\set{W \leq 0,\ t < t_0}}$ since $Z < W$ for $t < t_0$. But that is a contradiction with continuous expansion of $\set{W \leq 0}$, Lemma \ref{lem:continuousExpansion}. The same happens when $P_0$ lies in the interior of $\set{W \leq 0}$. Hence the conclusion of the lemma.
\end{proof}

The lemma states that the situation at point $P_0$ looks like the one shown in Figure \ref{fig:cross}.
\begin{figure}[t]
\centering
\includegraphics{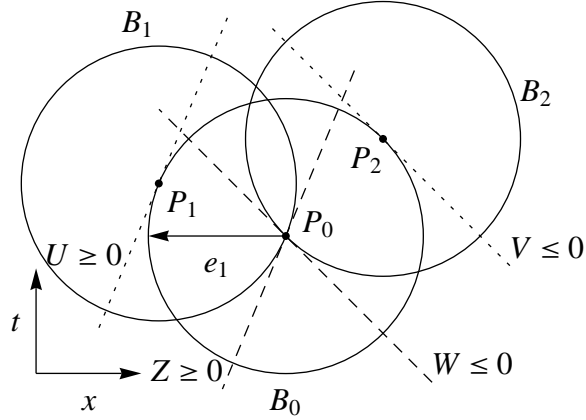}
\caption{Situation at $P_0$}
\label{fig:cross}
\end{figure}
More precisely, there exist $P_1 = (x_1, t_1) \in \partial\set{U \geq 0}$ and $P_2 = (x_2, t_2) \in \partial\set{V \leq 0}$ such that $U(P_1) = Z(P_0)$ and $V(P_2) = W(P_0)$. At $P_0$, $B_1 = B_r(x_1,t_1) = B_r(P_1)$ is an interior ball of $\set{Z \geq 0}$ and $B_2 = B_r(x_2,t_2) = B_r(P_2)$ is an interior ball of $\set{W \leq 0}$. Moreover, the ball $B_0 = B_r(x_0, t_0) = B_r(P_0)$ is an exterior ball of $\set{U \geq 0}$ at $P_1$ and $\set{V \leq 0}$ at $P_2$. 

Since $Z < W$ for $t < t_0$, the space projections of $B_1$ and $B_2$ are tangent:
\begin{equation*}
B_1 \cap B_2 \cap \set{t = t_0} = \edit{P_0}{\set{P_0}}.
\end{equation*}
Therefore the space projections of $P_0$, $P_1$ and $P_2$ are on a line and we can change the coordinates so that \edit{$P_0 = (0,0)$, $P_1 = (d_1 e_1, t_1 - t_0)$ and $P_2 = (-d_2 e_1, t_2 - t_0)$.}{
\begin{equation}
\label{eq:p012}
P_0 = (0,0),\quad P_1 = (d_1 e_1, t_1 - t_0)\quad \text{and}\quad P_2 = (-d_2 e_1, t_2 - t_0).
\end{equation}
}
\begin{lemma}
\label{lem:nonverticality}
Suppose that $Z(P_0)  = 0$. \edit{}{Then $(x_1,t_1) \neq (x_0, t_0 + r)$.} \edit{Then}{In other words,} the line $(P_0,P_1)$ cannot be vertical (i.e., be parallel to the time-direction) and in the future of $P_0$. \edit{In other words, when $t_1 > t_0$ the space projections of $P_1$ cannot be equal to the space projection of $P_0$ or equivalently, $t_1 \neq t_0 + r$. The same holds for $(P_0,P_2)$ when $W(P_0) = 0$.}{}
\end{lemma}
\begin{proof}
Suppose that $(P_1,P_0)$ is vertical, i.e. $x_1 = x_0$, and $t_1 = t_0 + r$. Thus there is $P_1' \edit{}{= (x_1', t_1)}\in D_r(P_1)$ such that $u(P_1') = 0$. Lemma \ref{lem:p0boundaries} implies that $P_0$ lies on the boundary of $\set{Z \geq 0}$. Therefore $u < 0$ in the interior of  $\Phi_r(P_0)$\edit{}{, defined in \eqref{eq:phir}}. This  and Lemma \ref{lem:continuousExpansion} yield that $u < 0$ in the interior of $D_r (P_1)$ and $P_1'$ lies on the boundary of $D_r(P_1)$\edit{}{, i.e. $\abs{x_1 - x_1'} = r$}. \edit{}{We will construct a classical supersolution in a parabolic neighborhood of $P_1'$ that crosses $u$ to get a contradiction.}  

\edit{Below we construct a classical supersolution of problem \eqref{eq:StefanProblem} in a parabolic neighborhood $E$ of $P_1'$, where }{First, let  $\om = \sqrt{r}/3$ and  $h = 1 + \edit{\sqrt{r}/3}{\om}$ and consider the set}
\begin{equation*}
E := \set{(x,t) : -1 \leq t \leq 0,\ 1- \edit{\sqrt{r}/3}{\om} - \om t \leq |x| \leq 1 + 2\edit{\sqrt{r}/3}{\om} - \om t}\edit{,}{.}
\end{equation*}
(see Figure \ref{fig:nonVerticalityZoom}). \edit{Note that $E \subset C'_\ve$ for small $r > 0$.}{}
\begin{figure}[t]
\centering
\includegraphics{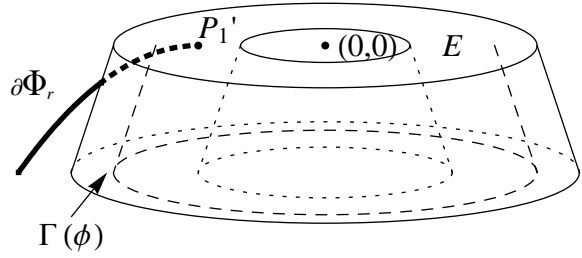}
\caption{Local situation at $P_1'$}
\label{fig:nonVerticalityZoom}
\end{figure}

We only show the construction for $n \geq 3$. The computation is similar when $n = 2$.
Thus consider a radially symmetric function $\phi(x,t)$ on $E$ of the form
\begin{equation*}
\phi(x,t) = \bcs
a ( -|x|^{2-n} + (h- \om t)^{2-n}) & \text{when } |x| \geq h - \om t,\\
b ( -|x|^{2-n} + (h- \om t)^{2-n}) & \text{when } |x| < h - \om t,
\ecs
\end{equation*}
where $a= \frac{\om}{2n - 4} h^{n-1}$ and $b > 0$ is a small constant to be chosen later.

$\phi$ then satisfies
\begin{align*}
\pth{\partial_t - \Delta} \phi = 
\partial_t \phi = c (n- 2) \om (h - \om t)^{1-n} > 0\edit{, \qquad}{\qquad \text{if }} |x| \neq h - \om t,
\end{align*}
where $c = a$ or $b$. The free boundary of $\phi$ is the shrinking sphere $\Gamma(\phi) = \set{|x| = h - \om t}$ and therefore the free boundary condition reads as
\begin{align*}
V_n = \om > a (n - 2) (h - \om t) ^{1-n} - \abs{D \phi^-} = \abs{D \phi^+} - \abs{D \phi^-}.
\end{align*}
\edit{Now define $\de = \phi(x, t)$ for any point on the exterior lateral boundary of $E$ (where $|x| = \edit{h - \om t + \sqrt{r} /3}{1 + 2 \om - \om t}$). Next choose $\ve > 0$ small enough so that $u < \de$ on the exterior lateral boundary of $E$, as was noted above.}{Hence $\phi$ is a classical supersolution on $E$. Now define 
\begin{align*}
\de = \min_{\substack{\abs{x}= 1 + 2\om - \om t\\ -1\leq t \leq 0}} \phi > 0.
\end{align*}
and choose $\ve$, $0 < \ve < r$, such that $u < \de$ on $B_{4\ve}(P_1')$. This can be done since $u(P_1') = 0$ and $u$ is upper semi-continuous. Let $E^\ve$ and $\phi^\ve$ be rescaled $E$ and $\phi$ under the parabolic scaling $(x,t) = (\ve \hat x + \ve x_1/r + (1 - \ve/r) x_1', \ve^2 \hat t + t_1)$. Note that $E^\ve \subset B_{4\ve}(P_1')$ and $\phi^\ve$ is still a supersolution in $E^\ve$.} 

\edit{Finally we finish the construction by choosing}{Finally, observe that the distance of the lateral boundary of $\Phi_r(P_0)$ from the time axis $\set{\hat x = 0}$ in the rescaled coordinates can be expressed as
\begin{align*}
|\hat x| = \sqrt{- 2 r \hat t - \ve^2 \hat t^2} + 1 \geq \sqrt{-r \hat t} +1 \qquad \text{if } -1 \leq \hat t\leq 0.
\end{align*}
Thus clearly $u < 0$ on $\set{\abs{\hat x} \leq 1 + 2 \om,\ \hat t = -1}$ and so we can choose} $b > 0$ small enough so that $\phi\edit{}{^\ve}(x,t) > u(x,t)$ on the part of parabolic boundary of $E\edit{}{^\ve}$ where $\phi\edit{}{^\ve} \leq 0$. This can be done as this is a closed subset of $\set{u < 0}$ and $u$ is upper semi-continuous. 

With above choices of $b$ and $\e$, $\phi\edit{}{^\ve}$ is a classical supersolution of \eqref{eq:StefanProblem} such that $\phi\edit{}{^\ve} > u$ on the parabolic boundary of $E\edit{}{^\ve}$ and $\phi\edit{}{^\ve}(P_1') < 0 = u(P_1')$.  Therefore $u - \phi\edit{}{^\ve}$ has a maximum inside $E\edit{}{^\ve}$, which yields a contradiction with $u$ being a viscosity subsolution.

The proof can be repeated \edit{for}{with} $-W$ in place of $Z$ and other appropriate substitutions.
\end{proof}

With the previous results,  it is now possible to establish the following lemma:
\begin{lemma}
\begin{equation*}
Z(P_0) = W(P_0) = 0.
\end{equation*}
\end{lemma}
\begin{proof}
It is enough to prove $Z(P_0) = 0$ because $W(P_0) = 0$ then follows by repeating the proof for $-W$.

Due to Lemma~\ref{lem:p0boundaries},  $Z(P_0) \geq 0$. Suppose $Z(P_0) = L > 0$. As in the proof of Lemma \ref{lem:nonverticality}, there is a point  $P_1' \in D_r(P_1)$ such that $u(P_1') = L$. 
First we discuss a couple of extreme cases:
\begin{enumerate}
    \item If $(P_1,P_0)$ is vertical with $P_1$ in the past of $P_0$, arguing as in the proof of Corollary \ref{cor:nonverticalityWhen0} we see that $(P_2,P_0)$ must be vertical with $P_2$ in the future. So either $W(P_0) = 0$, but that cannot happen by Lemma \ref{lem:nonverticality}, or $W(P_0) < 0$. For $W(P_0)<0$, one could argue as in (b) with $-W$ in the place of $Z$ (and other appropriate substitutions). 
    \item If $(P_1,P_0)$ is vertical with $P_1$ in the future, the fact that $P_0 \in \partial \set{Z \geq 0}$ implies $u < 0$ in the interior of $\Phi_r(P_0)$, therefore $u < 0$ in the interior of $D_r(P_1)$ by Lemma \ref{lem:continuousExpansion} and $P_1'$ must be on the boundary of $D_r(P_1)$. The rest of the argument is parallel to the one shown below.
\end{enumerate}   

Now suppose that $(P_0,P_1)$ is not vertical and $t_1 \leq t_0$. This allows us to find a small $\tau > 0$ and points $Q = (\tilde x, t_1)$ and $Q_\tau = (\tilde x, t_1 - \tau)$ such that $D_{h} (Q) \cap \set{u \geq 0} = \edit{P_1}{\set{P_1}}$ and $D_{h_\tau} (Q_\tau) \cap \set{u \geq 0} = \emptyset$, where $h = |P_1 - Q|$, $0 < h - h_\tau << 1$ (see Figure \ref{fig:Zless0}).
\begin{figure}[t]
\centering
\includegraphics{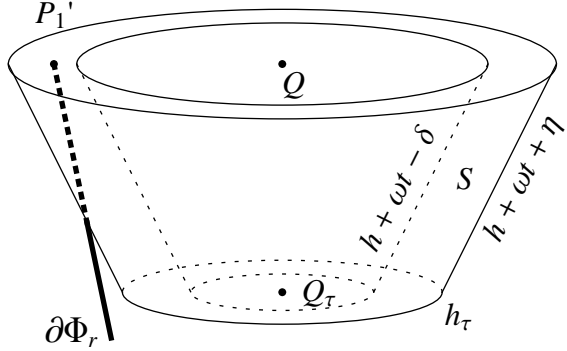}
\caption{Construction of a supersolution when $Z(P_0) > 0$.}
\label{fig:Zless0}
\end{figure}
For simplicity we translate coordinates so that $Q = (0,0)$. 

Denote $M = \max_Q u$. We will now construct a positive classical supersolution $\phi$ of \eqref{eq:StefanProblem} in 
$$
S = \set{(x,t) :  - \tau < t < 0,\ h + \om t - \de \leq |x| \leq h+\om t + \eta},
$$ 
where $\eta > 0$ and $\de > 0$ small are to be chosen and
$$
\om = \frac{h - h_\tau + \eta}{\tau},
$$
so that $\phi$ will cross $u$ at $P_1$, yielding a contradiction.

Consider a spatially radially symmetric function
\begin{equation*}
\phi(x,t) = f(|x| - \om t), \quad \qquad f(s) = -\frac{2c}{3K}\pth{-K s + a}^{3/2} + b,
\end{equation*}
where
$$
K = 2\pth{\frac{h-h_\tau/2}{\tau}+\frac{2n-2}{h_\tau}}>0
$$
and the constants $a, b \in \R$ and $c>0$ are to be chosen below. 

Observe that $g(s) = (f'(s)/c)^2= -Ks+a$.
Choose $0 < \eta < \half h_\tau$ small enough and $a \in \R$ so that $\half < g(s) < 1$ for all $h - \eta \leq  s \leq h + \eta$. It is also clear that we can find $c > 0$ and $b \in \R$ such that $0 < f(h) < L$ and $f(h + \eta) > M$. $f$ is continuous and so there is also $\de$, $0 <\de < \eta$ such that $f(h - \de) > 0$.

In particular, with this choice of $\eta$, $\de$, $a$, $b$ and $c$, $f(s)$ is positive and strictly increasing in $[h-\de, h + \eta]$. It follows that $\phi$ satisfies
\begin{align*}
(\partial_t - \Delta) \phi(x,t) &= -\bra{\pth{\om + \frac{n - 1}{\abs{x}}} f'(\abs{x} - \om t) + f''(\abs{x} - \om t)}\\ 
&=-\frac{1}{f'}\bra{\pth{\om + \frac{n - 1}{\abs{x}}} (f')^2 + f''f'}\\
&= -\frac{c^2}{f'}\bra{\pth{\om + \frac{n - 1}{\abs{x}}} g + \half g'}\\
& \geq 0
\end{align*}
since $\om + \frac{n - 1}{\abs{x}} \leq \ha{K}$ in $S$.
Therefore  $\phi$ is positive and strictly supercaloric in $S$.

To conclude, we note that $\phi > u$ on the parabolic boundary of $S$. This is clear from the fact that $\phi > 0$ in $S$ and $\phi > M$ on those parts of the parabolic boundary of $S$ that intersect $\set{u \geq 0}$. But $(u - \phi)(P_1) > 0$ and therefore $u$ crosses $\phi$ from below inside $S$. That is a contradiction with $u$ being a viscosity subsolution of \eqref{eq:StefanProblem}.

If $t_1 > t_0$, including the case when $(P_1,P_0)$ is vertical, the construction of a supersolution is more straightforward but we can also use the above construction with $h_\tau = h$. 
\end{proof}

\begin{corollary}
\label{cor:nonverticalityWhen0}
The vectors $P_0-P_1$ and $P_0-P_2$ cannot be vertical (i.e., be parallel to the time-direction). In other words, the space projections of $P_1$ and $P_2$ cannot be equal to the space projection of $P_0$.
\end{corollary}
\begin{proof}
Observe that if say $t_1 < t_0$ and the line $(P_1,P_0)$ is vertical, then also $(P_2,P_0)$ must be vertical and $t_2 > t_0$. This follows from the strict ordering $Z(\cdot, t) < W(\cdot, t)$ for all $t < t_0$ and $B_1$ being an interior ball of $\set{Z \geq 0}$ and $B_2$ being an interior ball of $\set{W \leq 0}$, i.e. $Z \geq 0$ in $B_1$ and $W \leq 0$ in $B_2$ which implies $B_1 \cap B_2 \cap \set{t <t_0} = \emptyset$. That is, $B_1 \cap B_2 = \edit{P_0}{\set{P_0}}$ whenever one of the balls is contained in $\set{t \leq t_0}$. But Lemma \ref{lem:nonverticality} asserts that $(P_2,P_0)$ cannot be vertical with $P_2$ in the future of $P_0$.
Similarly when $(P_2,P_0)$ is vertical with $t_2 < t_0$.
\end{proof}

Due to Corollary \ref{cor:nonverticalityWhen0}, the interior ball $B_1$ of $\{Z\geq 0\}$  has an interior normal $(e_1, m)$ at $P_0\in\partial B_1$ with $|m| <\infty$. Note that $m$ denotes the (outward) normal velocity of the space-time ball $B_1$ at $P_0$. Therefore, formally speaking, $m$ is smaller than the normal velocity $V_n$ of $\{Z\geq 0\}$ at $P_0$. Since $Z$ is a subsolution of \eqref{eq:StefanProblem} and since  $P_0\in\partial\{Z\geq 0\} = \partial\{Z<0\}$ , from the discussion in section 2 (see (V)) formally it must be true that
$$
m \leq |DZ^+| - |DZ^-|.
$$

Next lemma states that  this is indeed true in a weak sense.

\begin{lemma}
\label{th:zGradient}
\begin{equation*}
\liminf_{\la \ra 0+} \frac{Z(P_0+\la e_1,0)}{\la} + \liminf_{\mu \ra 0+} \frac{Z(P_0-\mu e_1, 0)}{\mu} \geq m.
\end{equation*}
\end{lemma}

\begin{proof}
Let us define
$$
a := \liminf_{\la \ra 0+} \frac{Z(P_0+\la e_1,0)}{\la}\quad \hbox{and}\quad  b := \liminf_{\mu \ra 0+} \frac{Z(P_0-\mu e_1, 0)}{\mu}.
$$

Observe that $b<0$.  Indeed by Corollary \ref{cor:nonverticalityWhen0}, we can find a smaller closed ball $B \subset B_2$ such that $B \cap \partial B_2 = P_0$ and $h > 0$ such that $B \cap \set{t =  t_0 - h} \neq \emptyset$. Since $Z<W$ for $t < t_0$, there is a smooth function $g\geq 0$ that solves the heat equation inside $B \cap \set{t_0 - h \leq t \leq t_0}$ such that $Z + g < W$ on $B \cap \set{t =  t_0 - h}$, $g > 0$ in $B^o \cap \set{t = t_0 - h}$ and $g = 0$ on $\partial B \cap \set{t_0 - h \leq t \leq t_0}$. As $Z$ is subcaloric and $W$ is supercaloric in $B \cap  \set{t_0 - h \leq t \leq t_0}$ it follows that $Z + g \leq W$ on $B \cap \set{t = t_0}$ and 
$$
0\geq \liminf_{\mu \ra 0+} \frac{W(P_0-\mu e_1,0)}{\mu}  \geq b+ \abs{D g\pth{P_0}} >b.
$$

Suppose $a+b<m$. Since $b<0$, $\ve > 0$ can be chosen sufficiently small so that
\begin{equation*}
a  + b  < m - 6 \ve\quad\hbox{and}\quad b+3\ve <0.
\end{equation*}

We will construct a supersolution of \eqref{eq:StefanProblem} constructed in a parabolic neighborhood $K$ of $P_1$, which will be specified later, which crosses $U$ from above, to yield a contradiction. We point out that, due to the two-phase nature of the problem, the construction of such barrier and the domain $K$ requires modifications from the one-phase version constructed in \cite{CV} and \cite{K1}.

From the definitions of $a$ and $b$, there are sequences $\la_k \ra 0$, $\mu_k \ra 0$ with $\la_k > 0$ and $\mu_k > 0$ such that
\begin{align*}
Z(x_0+\la_k e_1, t_0) &\leq (a + \ve) \la_k,\\
Z(x_0-\mu_k e_1, t_0) &\leq (b + \ve) \mu_k.
\end{align*}

In what follows, $\la \in \seq{\la}k$ and $\mu \in \seq{\mu}k$. Now denote $Q = P_0+(\la e_1, 0)$, $R = P_0+(-\mu e_1, 0)$ and $B_Q$, resp. $B_R$ the balls centered at $Q$, resp. $R$, with radius $r$.
The situation is illustrated in Figure \ref{fig:compP1s}.
\begin{figure}[t]
\centering
\includegraphics{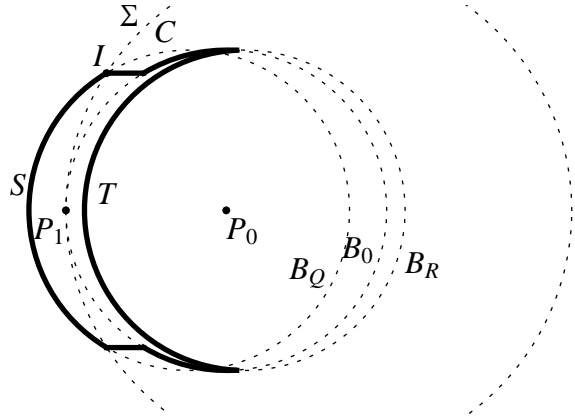}
\caption{The slice $\set{t = t_1}$}
\label{fig:compP1s}
\end{figure}

From the definition of $Z$, it follows that $U(x,t) \leq Z(y,s)$ for any $(y,s) \in B_r(x,t)$ and thus
\begin{equation*}
U \leq 
\bcs
Z(Q) \leq (a + \ve) \la & \tin B_Q,\\
Z(P_0) = 0 & \tin B_0,\\
Z(R) \leq (b+\ve) \mu & \tin B_R. 
\ecs
\end{equation*}
In fact, $U \leq 0$ on a slightly larger set. Because $Z \leq W$ at $t = t_0$ and $W \leq V(P_2) = 0$ in $B_2$, we can conclude that $Z \leq 0$ on $D = B_2 \cap \set{t = t_0}$. It means that $U \leq 0$ on $\Sigma$, the union of balls of radius $r$ with a center in $D$, 
\begin{equation*}
\Sigma = \bigcup_{(x,t) \in D} B_r(x,t).
\end{equation*}
Finally, we have an improved bound,
\begin{equation}
\label{eq:uBounds}
U \leq 
\bcs
 (a + \ve) \la & \tin B_Q,\\
0 & \tin \Sigma,\\
(b+\ve) \mu & \tin B_R. 
\ecs
\end{equation}

Now we will compare $U$ with a supersolution in a domain $K:=K_\tau^{\la,\mu}$. Fix $\tau >0$ small\edit{. $K_\tau^{\la,\mu}$ is the subset of $\set{t_1-\tau \leq t \leq t_1}$ with the shape }{
and define
\begin{align*}
K_\tau^{\la,\mu} = \set{t_1 - \tau \leq t \leq t_1} \cap \bigcup_{0\leq \gamma \leq \la} \pth{\partial B_Q\setminus \operatorname{int} \Sigma - \gamma (e_1, 0)} \cup B_0 \setminus \operatorname{int} B_R.
\end{align*}
In dimension $n=2$, the shape of this set is} depicted in Figure~\ref{fig:Kdomain} \edit{for $n = 2$.}{and the shape of its slice $\set{t = t_1}$} 
\begin{figure}[t]
\centering
\includegraphics{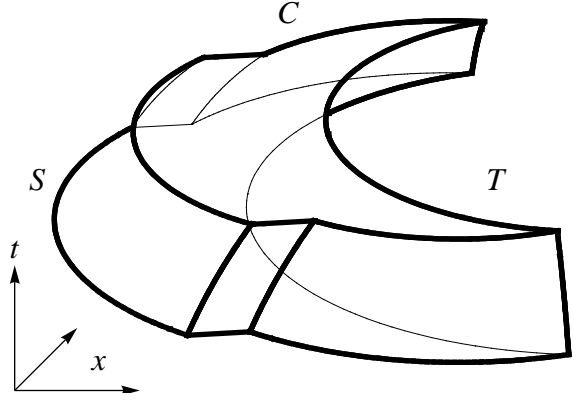}
\caption{Domain $K$ for $n = 2$}
\label{fig:Kdomain}
\end{figure}
\edit{The shape for each $t$}{} is depicted as the thick curve in Figure \ref{fig:compP1s}. The boundary of the slice is given by the arc  $S =\partial B_Q \setminus \Sigma$, the arc $T=\partial B_R \cap B_0$, and finally of the curve (surface) $C$ that connects $S$ and $T$. 
$C$ at any given time is composed of $C_1= \partial L\cap \Sigma \setminus B_R$, where $L$ is a cylinder with axis $e_1$ connecting $I =\partial B_Q\cap\partial \Sigma$  with $\partial B_0$, and $C_2= \partial B_0\setminus (B_R \cup L)$.

Let $\phi(x,t)= \phi(\rho,t)$, where $\rho = \abs{x - x_0}$. We need $U \leq \phi$ on the parabolic boundary of $K_\tau^{\la,\mu}$.
By Taylor expansion in $\lambda$ and $(t-t_1)$, it follows that the spatial distance of the point of intersection $I$ from $x_0$ can be found explicitly for each time $t$ as
$$
R_I(t)= d_1 + \la - m (t - t_1) + o\pth{\la} + o\pth{t- t_1} \leq d_1+\la-m(t-t_1),
$$
\edit{}{with $d_1$ from \eqref{eq:p012}.}
\edit{and}{Moreover,} $R_T(t) \leq |x - x_0|\leq R_S(t)$ in $K$, where 
\begin{equation*}
R_S(t) =  d_1 + \la - m (t - t_1) + o(t-t_1)\hbox{ and }R_T(t) = d_1 - \mu - m (t - t_1) + o(t-t_1).
\end{equation*}

Lastly,  the radius of the space-ball $B_0 \cap \set{t = t}$ is given by
\begin{align*}
R_1(t) = d_1 - m(t - t_1) + o(t - t_1).
\end{align*}

Now we are ready for the construction of the supersolution. Take $\al = a + 3 \ve$, $\be = b + 3 \ve$, $\tilde{R} = d_1 = R_1(t_1)$ and $\tilde{m} = \al + \be < m$. Let $\phi(x,t)$ be the function $\phi_{\tilde{R},\tilde{m}}^{\al,\be}(x - x_0, t - t_0)$ from Appendix. Let\edit{}{ $\tau > 0$, $\de > 0$ be the small constants and} $\Sigma_{\tau,\de}$ be the set from Appendix where $\phi$ exists and \edit{has the required properties}{is a radially symmetric classical supersolution that is "$\ve$-close" (in the sense of \eqref{eq:epsClose}) to linear functions with slopes $\al$ and $\beta$ on each side of the spherical free boundary of $\phi$ moving with speed $\tilde m$ and with radius $\tilde R$ at $t = t_1$}.

Let $\tilde{R}(t) = \tilde{R} - \tilde{m}(t-t_1)$ be the radius of the free boundary of $\phi$. We claim that for sufficiently small $\la, \mu$ and $\tau$ we have $U\leq \phi$ on the parabolic boundary of $ K_\tau^{\la,\mu}$.

To prove the claim, first note that for sufficiently small $\la, \mu > 0$
\begin{align*}
K_\tau^{\la,\mu} \subset \Sigma_{\tau, \de} = \set{(x,t): \tilde{R}(t) - \de \leq \abs{x - x_0} \leq \tilde{R}(t) + \de, \ t_1 - \tau \leq t \leq t_1}.
\end{align*}

We will compare $\phi$ with $U$ on two parts of the parabolic boundary.

{\it Part 1:} On \edit{}{lateral boundary} $\edit{\partial}{\partial_L} ( K_\tau^{\la,\mu} \cap \set{t_1 - \tau \leq t \leq t_1})$:

From (\ref{eq:uBounds}) we need
\begin{align}
\label{eq:phiBounds}
\phi(\rho,t) \geq \bcs
(a + \ve) \la & \text{for } R_I(t) \leq \rho \leq R_S(t),\\
0 & \text{for }  R_1(t) \leq \rho \leq R_I(t),\\
(b + \ve) \mu & \text{for } R_T(t) \leq \rho \leq R_1(t).
\ecs 
\end{align}
The second condition is satisfied as long as $\tilde R(t) \leq R_1(t)$ and that is true for all $t \in [t_1 - \tau, t_1]$ if $\tau$ is small enough.
Third condition is trivial for $\tilde R(t) \leq \rho \leq R_1(t)$. For $R_T(t) \leq \rho \leq \tilde R(t)$ we use the fact that
\begin{align*}
\phi(\rho,t) &\geq (\beta- \ve) \pth{\tilde R(t) - \rho} \\
&\geq (\beta- \ve) \pth{\tilde R(t) - R_T(t)} \\
&\geq (\beta- \ve) \mu > (b+\ve)\mu.
\end{align*}
Lastly the first condition in (\ref{eq:phiBounds}) holds since 
\begin{align*}
\phi(\rho,t) &\geq (\alpha - \ve) \pth{\rho - \tilde R(t)} \geq (\al - \ve) \pth{R_I(t) - \tilde R(t)}\\
&\geq (\alpha - \ve) (d_1 - m(t-t_1) + \la - \eta - d_1 + \tilde m(t-t_1)) \\
&= (\alpha- \ve) \pth{ (\tilde m - m)(t-t_1) + \la - \eta} \\
&\geq (\alpha - \ve) (\la - \eta) > (a+\e)\la
\end{align*}
if $\tau$ and $\la$ are small compared to $\e$ since $\eta = o(\la) + o(t - t_1)$.

{\it Part 2:}  On $K_\tau^{\la,\mu} \cap \set{t = t_1 - \tau}$:

From the bounds on $U$ in \eqref{eq:uBounds} we need 
\begin{align*}
\phi(\rho,t) \geq \bcs
(a + \ve) \la & \text{for } R_1(t_1 - \tau) \leq \rho \leq R_S(t_1 - \tau),\\
0 & \text{for }  R_T(t_1 - \tau) \leq \rho \leq R_1(t_1 - \tau).
\ecs 
\end{align*}

The second inequality can be satisfied by taking $\tau$ and $\mu$ small so that $\tilde R(t_1 - \tau) \leq R_T(t_1 - \tau)$. The first inequality is then satisfied by taking $\la$ small.

We demonstrated that $U \leq \phi$ on the parabolic boundary of $K_\tau^{\la,\mu}$ and $U \geq \phi$ at $P_1$ if $\la$, $\mu$ and $\tau > 0$ are small enough. Since $\phi$ is in fact a strict classical supersolution due to Appendix, we can perturb $\phi$ into $\tilde \phi$ near $P_1$ such that $\tilde \phi$ is still a classical supersolution with $U \leq \tilde \phi$ on the parabolic boundary of $K_\tau^{\la,\mu}$ and $U > \tilde \phi$ at $P_1$. This is a contradiction with the definition of $U$.

\end{proof}

We recall that $Z \leq W$ at $t = t_0$. Therefore Lemma~\ref{th:zGradient} as well as considering $g$ in the beginning of its proof  yields
\begin{align}
\label{eq:wGradients}
\limsup_{\la \ra 0+} \frac{W(P_0+\la e_1,0)}{\la} &+ \limsup_{\mu \ra 0+} \frac{W(P_0-\mu e_1, 0)}{\mu} \\
 &\geq \liminf_{\la \ra 0+} \frac{Z(P_0+\la e_1,0)}{\la} + \liminf_{\mu \ra 0+} \frac{Z(P_0-\mu e_1, 0)}{\mu} + \abs{D g(P_0)}\nonumber\\
&> m.\nonumber
\end{align}

Recall that strict ordering $Z < W$ for $t < t_0$ implies that $B_2$ has an exterior normal $(e_1, \hat m)$ with $\hat m \leq m$. With this result in mind, we can repeat the argument in Lemma~\ref{th:zGradient} to get a contradiction by proving that the quantity \eqref{eq:wGradients} must be \edit{less or equal than}{less than or equal to} $m$. 

This establishes the desired contradiction and we conclude that $Z$ and $W$ stay strictly separated for all times.

Note that Theorem~\ref{th:comparison} yields ordering only between strictly separated initial data. Therefore it does not directly yield the uniqueness result for viscosity solutions, even with the additional assumption $|\{u_0=0\}|=0$. However under some restrictions on initial data, uniqueness readily follows from Theorem~\ref{th:comparison}.
\begin{corollary}\label{cor:unique}
Let $u(x,t)$ be a viscosity solution of \eqref{eq:StefanProblem} with initial data $u_0$ and boundary data $\theta(x,t)=\theta(x) < 0$, where $u_0$ and $\ta$ satisfies one of the following conditions:
\begin{itemize}
\item[(a)] $u_0$ is strictly concave;
\item[(b)] $\theta(x,t) \equiv \theta<0$ and both $\Omega$ and $u_0$ is star-shaped: i.e. for sufficiently small $\e>0$,
$$
(1+\e)^{-1}\Omega \subset \Omega\hbox{ and } u_0(x) \leq u_0((1+\e)^{-1}x), \hbox{ where equality only holds at } x=0.
$$
\end{itemize}
Then $u(x,t)$ is unique.
\end{corollary}

\begin{proof}

1. Let us begin by mentioning that for any $a,b$ and $c>0$, if $u$ is a viscosity sub (super)-solution of \eqref{eq:StefanProblem}, then so is
$$
au(bx, ab^2 t+c).
$$

2. Now let $u$ and $v$ be two viscosity solutions of \eqref{eq:StefanProblem} with initial data $u_0$. In the case of (a), one can verify via barrier arguments that $v(x,t)\leq u_0(x)$ if $0\leq s<t$  with equality holding only at $\partial_L Q$.  In particular, for sufficiently small $\e$ and $\de$ such that $0 < \e \ll \de \ll 1$, we have $(1+\e)v(x,\de) < u_0(x)$ (recall that $\theta$ is negative). Therefore observation made in step~1  and  Theorem~\ref{th:comparison} yields
$$
(1+\e)v(x,(1+\e)t+\de) < u(x,t)\quad \hbox{ in } Q
$$
for any $0 < \e \ll \de \ll 1$. Now we can send $\de \to 0$ and $\e\to 0$ to obtain $v\leq u$. Similarly we can argue with $v$ and $u$ switched to obtain $u\leq v$, and therefore $u=v$.

In the second case, note that for any $\e>0$ and $\delta>0$, 
$$
(1+\delta)^{-1}u_0(x) < u_0((1+\e)^{-1}x)\hbox{ in } \Omega.
$$
Also, since $u_0$ is star-shaped, the minimum of $u_0$ (and $u$ and $v$) equals $\theta$. 
Since $u$ and $v$ are negative caloric functions near $\partial\Omega$, Hopf's lemma yields that the inward normal derivative of $u(\cdot,t)$ and $v(\cdot,t)$ on $\partial\Omega$ is strictly positive. 
In particular for any $T>0$ there exists $\delta = O(\e)$ such that  
$$
(1+\delta)^{-1}\theta < u((1+\e)^{-1}x,t) \hbox{ for } (x,t)\in\partial_LQ.
$$

Hence Theorem~\ref{th:comparison} again yields that, for sufficiently small $\e>0$,
$$
(1+\delta)^{-1} v(x, (1+\delta)^{-1} t) < u((1+\e)^{-1}x,(1+\e)^{-2}t) \quad \hbox{ in } Q. 
$$
Therefore, by sending $\e\to 0$, we can conclude that $v\leq u$. As before, by switching $v$ and $u$ in above argument one can conclude $v=u$.

\end{proof}
 For a general uniqueness result using the comparison principle we refer to section~\ref{sec:weakAreViscosity} in \cite{K1} where rather technical barrier arguments take place. We have not attempted to follow this approach since, as mentioned before, we are interested in the coincidence of the two generalized solutions for \eqref{eq:StefanProblem}.

\section{Correspondence between weak and viscosity solutions}
\label{sec:weakAreViscosity}

Recall that, due to \cites{LSU,F,M}, there exists a unique weak solution $h$ of problem \eqref{eq:StefanProblem}, defined in Definition~\ref{def:weakSolution}, with initial data $h_0$ and boundary data $\ta$. 
Let us define $u = \chi(h)$ to be the {\it weak temperature} solution of \eqref{eq:StefanProblem} associated with initial enthalpy function $h_0$, see \cite{F}.

\begin{lemma}
\label{th:weakIsViscosity}
Any weak temperature solution $u$ of \eqref{eq:StefanProblem} with initial data $u_0 = \chi(h_0)$ is a viscosity solution of \eqref{eq:StefanProblem}.
\end{lemma}

\begin{proof}
The analysis in \cite{CE} yields that weak temperature solutions are continuous for continuous $\chi(h_0)$. Moreover, a comparison principle holds for weak solutions. Indeed, if $h_1$ and $h_2$ are two weak solutions with $h_1\leq h_2$ on $\partial_P Q$, then $h_1 \leq h_2$ in $Q$. See \cite{M}, for instance. 

The comparison principle can be extended to accommodate weak sub and supersolutions. A function $h_1$ is a \emph{weak subsolution} if it satisfies \eqref{eq:weakSolution} in Definition~\ref{def:weakSolution} with $\geq$ instead of equality for all $\vp \in W^{2,1}_2(Q)$, $\vp \geq 0$, $\vp = 0$ on $\partial_L Q \cup \set{t = T}$. A function $h_2$ is a \emph{weak supersolution} if it satisfies  \eqref{eq:weakSolution} with $\leq$ instead of equality for all $\vp$ as above. A comparison principle still holds. If $h_1 \leq h_2$ on the parabolic boundary $\partial_P Q$ then $h_1 \leq h_2$ in $Q$. \edit{}{It is straightforward to check that the comparison principle on a cylindrical domain implies a comparison principle on a general parabolic neighborhood in the sense of Definition~\ref{def:parabolicNeighborhood}.}

A standard integration by parts also shows that a classical subsolution $\phi$ in Definition~\ref{def:classicalSubsolution} corresponds to a weak subsolution $\chi^{-1}(\phi)$, and also that a classical supersolution $\varphi$ is a weak supersolution $\chi^{-1}(\varphi)$: note that, by definition, $\phi$ and $\varphi$ have no mushy region and thus $\chi^{-1}\phi$ and $\chi^{-1}\varphi$ are well defined almost everywhere. We deduce from this that we can compare weak temperature solutions  with classical sub and supersolutions. Hence we conclude that weak temperature solutions are viscosity solutions of \eqref{eq:StefanProblem} in the sense of Definition~\ref{def2}.
\end{proof}

\begin{corollary}
There exists a viscosity solution of \eqref{eq:StefanProblem} for continuous $\chi(h_0)$ and $\theta$.
\end{corollary}

Unfortunately it is not clear to the authors how to find a corresponding weak solution for a given viscosity solution when the initial data has a fat zero set. However it is possible to identify the maximal and minimal viscosity solution with maximal and minimal weak temperature solutions (see Proposition~\ref{prop:maximal}). To prove this we first need a uniform control on the expansion of the sets $\set{u\leq 0}$ and $\set{u\geq 0}$ at $t=0$ in terms of the $L^\i$-norm of the initial data.

\begin{lemma}
\label{lem:expansionEstimate}
Let $u$ be a viscosity solution of \eqref{eq:StefanProblem} with initial data $u_0$. \edit{Further assume that $|\{u_0>0\}|=0$.}{} Then for any $r > 0$ the set $\set{u \leq 0}$ does not expand by more than $r$ before the time $t_r = c\edit{(n, M)}{} r^2$,\edit{}{ where $c = c(n,M)$ is a constant specified later in the proof, depending only on the dimension $n$ and }$M = \max_{\partial_P Q} {\abs{u_0}}$. In other words,
\begin{align*}
\set{x : u(x,t) \leq 0} \subset \set{x : \dist(x, \set{u_0 \leq 0}) \leq r} 
\end{align*}
for all $t \in [0, t_r]$. 

The same is true for $\set{u \geq 0}$.
\end{lemma}

\begin{proof}
We only prove the lemma for $\set{u \leq 0}$. Parallel argument holds for $\set{u \geq 0}$. 

Set 
\begin{align*}
A_r = \set{x : \dist(x, \set{u_0 \leq 0}) > r}.
\end{align*}
We need to show that $u(x, t) > 0$ in $A_r \times [0, t_r]$, for $t_r$ specified below. 

Thus choose $x_0 \in A_r$. Define the function 
\begin{equation*}
f(\rho) = \bcs
\frac{1}{\rho^{n-2}} - 1, & n \geq 3,\\
-\log \rho, & n = 2.
\ecs
\end{equation*}
Note that $f$ is decreasing and $f(1) = 0$. 
For simplicity we show the following computation for $n\geq 3$ only. The case $n = 2$ is similar.

Set $K = \frac{M}{\abs{f(2)}}$ and $v = \frac{8 K (n-2)}{r}$.
We will construct a classical subsolution of (\ref{eq:StefanProblem}) in the set
\begin{align*}
C = \set{(x,t) : |x- x_0| < r - v t, \ 0 \leq t \leq t_r}
\end{align*}
that is below $u$ on the parabolic boundary of $C$ and is positive at $x_0$ for $0 \leq t \leq t_r$.

Define the radially symmetric function $\phi = \phi(x,t)$ in $C$ as
\begin{equation*}
\phi(x,t) = \bcs
K f\pth{\frac{2\abs{x - x_0}}{r - v t}}, & 2|x - x_0| > r - v t,\\
\tilde \phi(x,t), & 2|x-x_0| < r - v t,
\ecs
\end{equation*}
where $\tilde \phi(x,t)$ is a solution of the heat equation in the domain 
$$
C_\half = \set{(x,t) : |x - x_0| < \half\pth{r - v t}, \ t \in (0, t_r)}
$$
with smooth positive initial data $\tilde \phi(x,0) < u(x,0)$ in the interior of $D = \set{x : \abs{x - x_0} < r / 2}$ and boundary data $0$. 

 The modulus of the gradient $\abs{D\phi^-}$ on the free boundary $\Gamma(\phi)$ is 
\begin{equation*}
\abs{D \phi^-} = K \abs{f'\pth{1}} \frac{2}{r - vt} =  \frac{2 K (n - 2)}{r - v t}.
\end{equation*}
For $t < t_r = \frac{r}{2 v} = c(n, M) r^2$,
\begin{equation*}
\abs{D \phi^-} < \frac{4 K (n - 2)}{r} = \ha v = - V_n.
\end{equation*}
Therefore $\phi$ is a classical subsolution of \eqref{eq:StefanProblem} in the set $C$. Also, $\phi < u$ on the parabolic boundary $\partial_P C$. \edit{Theorem~\ref{th:comparison}}{The definition of viscosity subsolution then} implies \edit{$\phi < u$}{$\phi \leq u$} in the set $C$ and therefore $u > 0$ at $x_0$ for $0 \leq t \leq t_r$.
\end{proof}

\begin{proposition}\label{prop:maximal}
 There exist maximal and minimal viscosity solutions $U_1(x,t)$ and $U_2(x,t)$ of \eqref{eq:StefanProblem} satisfying the following: 
$$
U_2\leq u\leq U_1\quad\hbox{ for any viscosity solution } u \hbox{ of \eqref{eq:StefanProblem}. }
$$ 
Furthermore, $U_1$ is the unique weak temperature solution of \eqref{eq:StefanProblem} with the \edit{intial}{initial} enthalpy
$$h^{max}_0:= \left\{\begin{array}{ll}
u_0&\hbox{ if }u_0\geq 0,\\
u_0-1 &\hbox{ if }u_0<0,
 \end{array}\right.
$$
and $U_2$ is the unique weak temperature solution of \eqref{eq:StefanProblem} with the initial enthalpy 
$$h^{min}_0:=\left\{\begin{array}{ll}
u_0&\hbox{ if }u_0>0, \\
 u_0-1&\hbox{ if }u_0\leq 0.\end{array}\right.
$$ 
\end{proposition}

 \begin{proof}

1. Take an approximation $w^\e_0 < u_0 < v^\e_0$ such that the zero set of the approximation has zero measure.
This can be done, for example, with simply $w^\e_0 = u_0-\e$ and $v^\e_0 = v_0+\e$ for along a sequence of \edit{$\e_0\to 0$}{$\e\to 0$}. Indeed the function 
$F(\e):=|\{u_0>0\}|-|\{u_0>\e\}|$ is a monotone, bounded function. therefore it can only have countable number of jumps. Except at these jumps, $|\{u_0=\e\}|=0$. 

Let us choose a viscosity solution $v^\e$ (and $w^\e$) of \eqref{eq:StefanProblem} respectively with initial data $v^\e_0$ (resp. $ w^\e_0$) and lateral boundary data $\theta+\e$ (resp. $\theta-\e$).
In particular (excluding a countable number of $\e$) the functions $v^\e_0$ and $w^\e_0$ have no mushy region, and therefore $h^{1,\ve}_0:=\chi^{-1}(w^\e_0)$ and $h^{2,\ve}_0:=\chi^{-1}(v^\e_0)$ are well-defined almost everywhere, hence the corresponding weak solutions $h^{1,\e}$ and $h^{2,\e}$ exist. Now, by Lemma~\ref{th:weakIsViscosity}, the weak temperature solution $v^\e:= \chi(h^{1,\e})$ (resp. $w^\e:=\chi(h^{2,\e})$) is a viscosity solution of \eqref{eq:StefanProblem} with initial data $v^\e_0$ (resp. $w^\e_0$).
Let $u$ be a viscosity solution of \eqref{eq:StefanProblem} with initial data $u_0$. Let $\hat u$ be the unique weak temperature solution of \eqref{eq:StefanProblem} with boundary data $h_0:=\chi^{-1}(u_0)$.
For convenience, let us denote the initial and lateral boundary data $g:= (u_0,\theta)$.

2. The proof of the following lemma will be presented in the Appendix:
\begin{lemma}\label{convergence}
\edit{}{The functions} $v^\e$ (or $w^\e$), along a subsequence, uniformly converge to a continuous function $u$ with boundary data $g$ .
\end{lemma}
Due to the stability of viscosity solutions, it is straightforward to verify that $u$ is a viscosity solution of \eqref{eq:StefanProblem} with boundary data $g$.  

On the other hand, due to stability result of \cite{M}, for each time $t>0$  the weak soution $h^{1,\e}(\cdot,t)$ converges to $h(\cdot,t)$ in $L^1$-norm, where  $h(\cdot,0) = h^{max}_0$.

Using the uniform convergence of $v^\e:= \chi(h^{1,\e})$, it is easy to check that $u=\chi(h)=U_1$ a.e.

3. Lastly, Theorem~\ref{th:comparison} yields that any viscosity solution $u$ of \eqref{eq:StefanProblem} with initial data $u_0$ satisfies \\ $w^\e < u < v^\e$, and thus taking $\e\to 0$ we obtain $U_2\leq u\leq U_1$.

 \end{proof}

The following Theorem is obtained as a corollary of Propostion~\ref{prop:maximal}

\begin{thm}
\label{th:viscositySolutionIsUnique}
The viscosity solution $u$ of \eqref{eq:StefanProblem} with $|\{u_0=0\}|=0$ coincides with the weak temperature solution and it is unique.
\end{thm}

\section*{Appendix}

First we construct ``two-phase'' test functions for the barrier arguments in the proof of the comparison principle, Theorem~\ref{th:comparison}.

\begin{proof}[A. Test functions]
 Here we provide a construction of a class of radially symmetric classical sub- and supersolutions for that purpose.
We will be working in spherical coordinates and thus we shall denote $\rho = \abs{x}$.

\newcommand{\sign}{\operatorname{sign}}

Let $\al \neq 0,\ \be \neq 0$, $\ve >0$, $R > 0$ and $m$ be given parameters satisfying 
\begin{align}
\label{eq:conditionOnParameters}
\sign \al \neq \sign \be \qquad \text{ and } \qquad \al + \be \neq m \sign \be.
\end{align}
We show below that one can find constants $\tau > 0$, $\de > 0$ and $q$ such that a test function $\phi(x,t)$
of the form 
\begin{equation*}
\phi(x,t) = \phi_0 (\abs{x} - R - m t).
\end{equation*}
with
\begin{equation*}
\phi_0(s) = 
\bcs
\al s + q \frac{s^2}{2} & s \geq 0,\\
-\be s + q \frac{s^2}{2} & s < 0,
\ecs
\end{equation*}
defined in the domain 
\begin{equation*}
\Sigma_{\tau,\de} = \set{(x,t): R + m t - \de \leq \abs{x} \leq R + m t + \de,\ -\tau \leq t \leq 0},
\end{equation*}
is a classical subsolution (or supersolution) of \eqref{eq:StefanProblem} on $\Sigma_{\tau,\de}$, see Figure~\ref{fig:testFunction1}. Moreover, $\phi(x,t)$ is ``$\ve$-close'' to a linear function of $\rho$ on both sides of the free boundary $\Gamma(\phi)$, see \eqref{eq:epsClose}. 

\begin{figure}[t]
\centering
\includegraphics{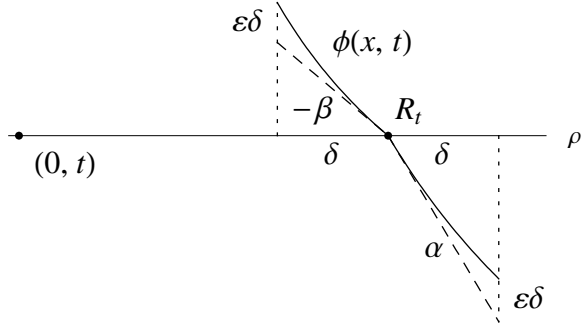}
\caption{Test function $\phi(r,t)$ at time $t$}
\label{fig:testFunction1}
\end{figure}

We shall denote $R_t = R + mt$ since the set $\Gamma(\phi) = \set{(x,t) : r = R + m t}$ is the free boundary of $\phi$. Observe that on $\Gamma(\phi)$, we have
\begin{align*}
\abs{D\phi^+} = \max\pth{\al, \be}, \qquad \abs{D\phi^-} = - \min\pth{\al,\be}.
\end{align*}
and thus 
\begin{align}
\label{eq:alphaBetaGradients}
\abs{D\phi^+} - \abs{D\phi^-} = \al + \be.
\end{align}

We can also deduce the direction of the outer normal of $\set{\phi > 0}$, and hence the sign of the normal velocity $V_n$, from the signs of  $\al$ and $\be$ due to \eqref{eq:conditionOnParameters}. As $\Gamma(\phi)$ is the space ball expanding with speed $m$, we recognize two situations:
\begin{align}
\label{eq:signOfm}
   V_n = \bcs 
-m & \text{when } \be < 0,\\
m & \text{when } \be > 0.
\ecs
\end{align}

Now we can express the free boundary condition in the light of \eqref{eq:alphaBetaGradients} and \eqref{eq:signOfm} as
\begin{enumerate}
   \item $m \sign \be \geq \al + \be$ for a classical supersolution,
	\item $m \sign \be \leq \al + \be$ for a classical subsolution.
\end{enumerate}
Suppose that the condition in (a) is satisfied and we are thus constructing a classical supersolution. In this case, the function $\phi$ has to satisfy also $(\partial_t - \Delta) \phi \geq 0$ in $\set{\phi \neq 0}$.

Straightforward computation yields
\begin{align*}
\at{\pth{\partial_t - \Delta} \phi(x,t)}{\rho = R_t^+} &= - \al \pth{m + \frac{n-1}{R_t}} - q,\\
\at{\pth{\partial_t - \Delta} \phi(x,t)}{\rho = R_t^-} &= \beta \pth{m + \frac{n-1}{R_t}} - q.
\end{align*}
Continuity allows us to choose small $\tau > 0$ and $\de > 0$ and a suitable constant $q$ such that
\begin{align}
\label{eq:conditionPhi}
\pth{\partial_t - \Delta} \phi(x,t) \geq 0 \qquad \text{in } \Sigma_{\tau,\de} \setminus \Gamma(\phi)
\end{align}
and
\begin{align}
\label{eq:epsClose}
\begin{aligned}
\abs{\frac{\phi(x,t)} {\rho - R_t} - \al} &< \ve,\qquad \text{when } \rho > R_t,\\ 
\abs{\frac{\phi(x,t)}{R_t - \rho} - \be} &< \ve, \qquad \text{when } \rho < R_t,
\end{aligned}
\end{align} 
in $\Sigma_{\tau,\de}$.

Thus let $\phi_{R,m}^{\al,\be}$ denote the classical supersolution constructed above with the parameters $R$, $m$, $\al$ and $\be$. Also let $\de\pth{\phi_{R,m}^{\al,\be}, \ve}$ and $\tau\pth{\phi_{R,m}^{\al,\be}, \ve}$ denote the constants from above.

The case when the condition in (b) is satisfied and we are constructing a subsolution is analogous. The constants $\tau> 0 $, $\de > 0$ and $q$ are chosen in such a way that the function $\phi$ satisfies $\pth{\partial_t - \Delta} \phi(x,t) \leq 0$ in \eqref{eq:conditionPhi}. And correspondingly, let $\phi_{R,m}^{\al,\be}$ be the constructed classical subsolution.
\end{proof}

\vspace{30pt}

We next present the proof of the lemma used in the proof of Proposition~\ref{prop:maximal}.

\begin{proof}[B. Proof of Lemma~\ref{convergence}]

We will only prove the Lemma for $v^\e$. The argument for $w^\e$ is parallel. Let $u^k := v^{1/k}$ and let $g^k$ the corresponding initial and lateral boundary data of $v^{1/k}$.

 \cite{CE} provide us with a uniform bound on the $W^{1,0}_2(Q)$ norm and with a uniform modulus of continuity $\omega_d$ for $v^\e$ in the domains away from $\partial_P Q$:
\begin{align*}
Q_d = \set{(x,t) \in Q : \dist((x,t), \partial_P Q) > d}, \qquad d > 0.
\end{align*}
Therefore we need a uniform modulus of continuity on $\partial_P Q$ to be able to extract a subsequence of $u^k$ converging uniformly. We will obtain this by comparing $u^k$ to solutions of the heat equation near the parabolic boundary of $Q$.

First find $k_0$ for which $g^{k_0} < 0$ on $\partial_L Q$.
Theorem~\ref{th:comparison} implies $u^k$ decreases in $k$. In particular 
$$
\set{u^{k+1} \geq 0} \subset \set{u^{k} \geq 0} \subset \set{u^{k_0} \geq 0}.
$$ 
Since $\set{u^{k_0} \geq 0}$ is a closed set and $g^k < 0$ for $k \geq k_0$, there exists $\eta > 0$ such that
\begin{align*}
\set{u^{k} \geq 0} \subset Q^\eta := \set{(x,t) \in Q : \dist(x, \partial\Omega) > \eta} \fall k \geq k_0.
\end{align*}
Note that $Q^\eta := \Omega^\eta \times (0, T)$ where
\begin{align*}
\Omega^\eta = \set{x \in \Omega : \dist(x, \partial \Omega) > \eta}.
\end{align*}

Define $\psi^k_+$, resp. $\psi^k_-$, to be the solution of the heat equation in $Q$ with boundary data $g^{k+}$, resp. $-g^{k-}$. Here $g^{k+}$ and $g^{k-}$ are the positive part and the negative part of $g^k$, respectively.

The maximum principle for the heat equation yields
\begin{align}
\label{eq:heatSubSuper}
\psi^k_-  \leq u^k \leq \psi^k_+ \qquad \text{in } Q.
\end{align}

A uniform modulus of continuity for $u^k$ on the lateral boundary $\partial_L Q$ then follows from comparing $u^k$ with $\vp^k$ that solves the heat equation
\begin{align*}
\bcs
\vp_t - \Delta \vp = 0,& \text{in } Q \setminus Q^\eta,\\
\vp = g^k, &\text{on } \partial_L Q,\\
\vp(x, 0) = g^k(x, 0) \zeta(x), &\text{on } \Omega \setminus \Omega^\eta.
\ecs  
\end{align*}
Here $\zeta(x)$ is a smooth function on $\Omega \setminus \Omega^\eta$ such that $\zeta \in [0,1]$, $\zeta = 1$ on $\Omega \setminus \Omega^{\eta/2}$ and $\zeta = 0$ on $\partial \Omega^\eta$.

Observe that \eqref{eq:heatSubSuper} and comparison for the heat equation implies that 
\begin{align*}
\psi^k_-  \leq u^k \leq \vp^k \qquad \text{on } Q \setminus Q^\eta.
\end{align*}
Since $\psi^k_- = u^k = \vp^k = g^k$ on $\partial_L Q$, the uniform continuity follows.

The proof for $t = 0$ is more involved. We split $u^k = u^{k+} - u^{k-}$ and consider each of $u^{k\pm} \geq 0$ separately.

Choose $r > 0$ small and fix $k$. Let $A^k = \set{x \in \Omega : g^k >0 }$. Define 
\begin{align*}
A_r^k = \set{x \in A : \dist(x, \partial A \cap \Omega) > r}.
\end{align*}
 By Lemma \ref{lem:expansionEstimate}, there is $t_r > 0$ such that $A_r^k \subset \set{x:u^k(x,t) > 0}$ for all $t < t_r$. Moreover, $t_r$ can be chosen independent of $k$. Construct a solution $\vp^k$ to the heat equation in $A_r^k \times (0, t_r)$ with boundary data 0 on $(\partial A^k_r \cap \Omega) \times (0, t_r)$ and $g^k(x, t) \zeta^k(x)$ on the rest of the parabolic boundary of $A^k_r \times (0, t_r)$. $\zeta^k$ is chosen to be a smooth cutoff function with values in $[0,1]$ on $A^k_r$ such that $\zeta^k = 1$ on $A^k_{2r}$ and $\zeta^k = 0$ on $\partial A^k_r \cap \Omega$. It is also chosen in such a way that $\no{\zeta^k}_{C^{0,1}}$ is uniformly bounded in $k$. Define $\vp_r^k = 0$ in the complement of $A^k_r \times (0,t_r)$. The maximum principle for heat equation then yields
\begin{align*}
\vp_r^k \leq u^{k+} \leq \psi^k_+ \qquad \text{in } \Omega \times [0,t_r].
\end{align*}
Thus the standard parabolic estimates for the heat equation (\cite{LSU}) imply
\begin{align*}
\abs{u^{k+}(x_1, t_1) - u^{k+}(x_0,0 )} &= \abs{u^{k+}(x_1, t_1) - g^{k+}(x_0)} \\
&\leq \om_g(2r) + \om_H\pth{\abs{x_1-x_0}^2 + t_1},
\end{align*}
for all $t_1 < t_r$, $x_0, x_1 \in \Omega$, modulus of continuity $\om_H$ independent of $k$ and $r$, where $\om_g$ is a modulus of continuity of $g$. 
This implies uniform modulus of continuity for $u^{k+}$ at $ t = 0$. 

A similar consideration yields a uniform modulus for $u^{k-}$ at $t = 0$. Those together give us a uniform modulus of continuity for $u^k$.

Now as in \cite{CE}, we can find a subsequence $\set{u^{k_j}}_{j=1}^\i$ such that
\begin{align*}
\beta\pth{u^{k_j}} &\ra \beta\pth{\tilde u} & \text{weakly in } L^2(Q),\\
D u^{k_j} &\ra D\tilde u & \text{weakly in } L^2(Q),\\
u^{k_j} &\ra \tilde u & \text{uniformly on } \cl{Q}.
\end{align*}
\end{proof}

\noindent\textbf{Acknowledgments:} I. Kim and N. Pozar are supported by NSF-DMS 0700732.


\end{document}